\documentclass[a4paper]{amsart}
\usepackage{amssymb} % needed for \gtrapprox

\usepackage[utf8]{inputenc}
\usepackage[T1]{fontenc} % correctly encodes fonts in the output

\usepackage{microtype}
\usepackage{mathtools}
   \mathtoolsset{centercolon} % makes ' := ' nice
   
 \usepackage[shortlabels]{enumitem}

\usepackage[hidelinks]{hyperref} % no ugly link borders
   \hypersetup{final} % to have those links while draft option is activated

\newtheorem{theorem}{Theorem}[section]
\newtheorem{lemma}[theorem]{Lemma}
\newtheorem{corollary}[theorem]{Corollary} % added
\newtheorem{proposition}[theorem]{Proposition} % added

\theoremstyle{remark}
\newtheorem{remark}[theorem]{Remark}

\numberwithin{equation}{section}

    \newcommand*{\RR}{\mathbb{R}}
    \newcommand*{\NN}{\mathbb{N}}

    \newcommand*{\ind}[1]{\mathbf{1}_{#1}}
    \newcommand*{\indbr}[1]{\ind{\{#1\}}}

    \newcommand*{\alphaA}{\alpha_0}
    \newcommand*{\betaA}{\beta_0}
    \newcommand*{\betaB}{\beta_1}
    \newcommand*{\aBA}{a_*} % transition point for p>2
    \newcommand*{\betaBA}{{\beta}_*} % transition point for p>2, aux. parameter
    \newcommand*{\alphaBbis}{\alpha_2}
    \newcommand*{\betaBbis}{\beta_2} 
\newcounter{proofincases}
\newcommand\nextcase{Case \stepcounter{proofincases}\arabic{proofincases}}
\newcommand\nextcaselabel[1]{Case \refstepcounter{proofincases}\label{#1}\arabic{proofincases}}

\begin{document}

\title[Hardy's operator minus identity and power weights]{Hardy's operator minus identity\\ and power weights}
\author[M. Strzelecki]{Micha\l{} Strzelecki}
\address{Institute of Mathematics, University of Warsaw, Banacha 2, 02--097 Warsaw, Poland.}
\email{michalst@mimuw.edu.pl}
\thanks{Research partially
supported by the National Science Centre, Poland, via the Preludium grant no.\ 2015/19/N/ST1/00891.}

\begin{abstract}
Let $H$ be the Hardy operator and $I$ the identity operator acting on functions on the real half-line. 
We find optimal bounds for
the operator $H - I$ in the setting of power weights and the cases of positive decreasing functions, positive functions, and
general functions. 
As a byproduct, we obtain some results about the optimal relations between the norms of $H$  and its dual.
\end{abstract}

\date{September 11, 2019}



\maketitle

\section{Introduction}

For a locally integrable function $f\colon [0,\infty) \to \RR$ the Hardy operator is defined by 
\begin{equation*}
  H f(t) := \frac{1}{t} \int_0^t f(s) ds, \quad t>0.
\end{equation*}
Hardy's inequality \cite[Chapter 9]{MR0046395} states that for $1\leq p < \infty$ and $a<p-1$, 
\begin{equation}
\label{eq:Hardy}
\int_0^\infty |Hf(t)|^p t^a dt \leq \Bigl(\frac{p}{p-1-a} \Bigr)^p \int_0^\infty |f(t)|^p t^a dt.
\end{equation}
Necessary and sufficient conditions for this inequality to hold with $t^a$  replaced by more general weights
were found by Muckenhoupt \cite{MR0311856}.
Subsequently, Ari\~{n}o and Muckenhoupt \cite{MR989570} and Sawyer \cite{MR1052631}  studied the boundedness of $H$ on a class of classical Lorentz spaces which led them to consider the general weighted inequality
but restricted to positive decreasing functions.

In this paper we are interested in optimal constants in weighted bounds for the operator $H-I$ acting on positive decreasing functions, positive functions, and general functions. Let us mention two important functional analytic motivations for this problem.  The case of positive decreasing functions is obviously connected to the concept of monotone rearrangements and is motivated by results about normability and embeddings of function spaces (see, e.g., \cite{MR621018,MR2182593,MR2411048}). On the other hand, the case of general functions in $L^p([0,\infty))$ corresponds to the study of the Beurling--Ahlfors transform acting on radial functions (see, e.g., \cite{MR2595549,MR3558516} and the references therein) and is interesting even in the unweighted setting.
Before presenting the contribution of the current paper, let us discuss relevant prior results.

Using a multinomial theorem Kruglyak and Setterqvist \cite{MR2390520} proved that for \emph{integer} $p\geq 2$ and any positive decreasing $f\colon[0,\infty)\to \RR$,
\begin{equation*}
 \label{eq:KS08}
 \int_0^\infty |Hf(t) -f(t)|^p dt \leq \frac{1}{p-1} \int_0^\infty |f(t)|^p dt.
\end{equation*}

Boza and Soria \cite{MR2747011} extended this to all values of $p\in[2,\infty)$ by providing a more general result. Namely, they proved that if $p\geq 2$ and $w$ is a weight in the Ari\~{n}o--Muckenhoupt class  $B_p$,
i.e.,
\[
 \|w\|_{B_p} \coloneqq \sup_{t>0} \frac{t^p \int_t^\infty \frac{w(s)}{s^p} ds}{\int_0^t w(s) ds}<\infty,
\]
and moreover
\begin{equation}
 \label{eq:BS11-condition}
 t^{p-1} \int_t^\infty \frac{w(s)}{s^p} \leq \frac{w(t)}{p-1} \quad \text{a.e.},
\end{equation}
then for any positive decreasing $f\colon[0,\infty)\to \RR$,
\begin{equation}
 \label{eq:BS11-general}
 \int_0^\infty |Hf(t) -f(t)|^p w(t) dt \leq \|w\|_{B_p} \int_0^\infty |f(t)|^p w(t) dt.
\end{equation}
The condition \eqref{eq:BS11-condition} holds, e.g., for decreasing weights.
In particular, for $p\geq 2$, $a\in (-1,0]$, and any positive decreasing $f\colon[0,\infty)\to \RR$,
\begin{equation*}
  \int_0^\infty |Hf(t) -f(t)|^p t^a dt \leq \frac{1+a}{p-1-a} \int_0^\infty |f(t)|^p t^a dt.
\end{equation*}

Boza and Soria noticed that \eqref{eq:BS11-general} holds also for $p=1$ and $w\in B_1$ (without any further assumptions), but fails for $1<p<2$ (e.g., if $w\equiv 1$) as well as for 
$p\geq 2$ when the monotonicity hypothesis on $w$ is dropped (e.g., if $w(t) = t^a$ with $0<a<p-1)$.

For $1 < p \leq 2$ the best constant in the estimate \eqref{eq:KS08}, still for positive decreasing functions,  is equal to $\frac{1}{(p-1)^p}$, as shown by Kolyada \cite{MR3180926}. 
His approach relies on reducing the problem to proving that for $1<p\leq 2$ and any positive function $g\colon [0,\infty)\to\RR$,
\begin{equation*}
 %\label{eq:K12-H-H^*}
  \int_0^\infty |Hg(t)|^p dt \leq \frac{1}{(p-1)^p} \int_0^\infty |H^*g(t)|^p dt.
\end{equation*}
where
\[
 H^* g(t) \coloneqq \int_t^\infty \frac{g(s)}{s} ds 
\]
is the dual of $H$. Here one does not have to assume that $g\in L^p([0,\infty))$
(cf.\ Section~\ref{sec:Hstar} below).

In an independent series of papers regarding the  action of the Beurling--Ahlfors transform on radial functions, the operator $H-I$ acting on the whole space $L^p([0,\infty))$ was studied.
Ba\~{n}uelos and Janakiraman \cite{MR2595549} proved that for $1<p\leq 2$ and general $f\colon [0,\infty) \to \RR$,
\begin{equation}
 \label{eq:BJ09}
  \int_0^\infty |Hf(t) -f(t)|^p dt \leq \frac{1}{(p-1)^p} \int_0^\infty |f(t)|^p dt;
\end{equation}
alternative proofs were given in \cite{MR3018958} and \cite{volberg}.
The best constant in the complementary range $p>2$ was found by the author in \cite{MR3558516}:  the expression is more complicated and involves the root of an equation. The proof  in \cite{MR2595549} is based on properties of stretch functions, while \cite{MR3018958,volberg,MR3558516} exploit ideas connected to the Bellman function technique and Burkholder's work on martingale inequalities.

Finally, Boza and Soria \cite{MR3868629} observed recently that for $p\geq 2$ and any positive $f\colon[0,\infty)\to\RR$,
\begin{equation*}
% \label{eq:BS19-positive}
   \int_0^\infty |Hf(t) -f(t)|^p dt \leq \int_0^\infty |f(t)|^p dt.
\end{equation*}

In the current paper we find the best constants in the estimate
\begin{equation*}
    \int_0^\infty |Hf(t) -f(t)|^p t^a dt \lesssim \int_0^\infty |f(t)|^p t^a dt, \quad f\in \mathcal{C},
\end{equation*}
for the whole range of admissible parameters $p$ and $a$, when $\mathcal{C}$ is one of the following classes: positive decreasing functions, positive functions, general functions. 
To  this end we adapt the techniques of \cite{MR3558516}. This enables us to provide a uniform framework for all inequalities and kill several birds with one stone.  Our approach gives relatively elementary---and in many cases quite short---proofs. It also  demonstrates how and why the optimal constant changes when we change the values of $p$ and $a$ or modify the class of functions involved.

\section{Results}

For given $1\leq p < \infty$  and $a < p-1 $ denote by $C_{p,a}$ (resp.\ $B_{p,a}$, resp.\ $A_{p,a}$), the smallest constant $K=K(p,a)$ for which the inequality   
\begin{equation*}
\int_0^\infty |Hf(t) -f(t)|^p t^a dt \leq K^p \int_0^\infty |f(t)|^p t^a dt
\end{equation*}
is satisfied for all (resp.\ all positive, resp.\ all positive and decreasing)  functions $f\colon [0,\infty)\to \RR$ for which the right-hand side is finite.
Note that in the case  of $A_{p,a}$ we can consider only $a\in(-1,p-1)$ since for $a\leq -1$  and non-trivial positive decreasing $f$ the integral on the right-hand side is always infinite.

For $1\leq p<\infty$, $a <p-1$, and $\alpha < \frac{p-1-a}{p} < \beta$ denote
 \begin{equation}
 \label{eq:def_k}
k_{p,a}(\alpha, \beta) \coloneqq \frac{(\beta-\frac{p-1-a}{p})|\alpha-1|^p + (\frac{p-1-a}{p}-\alpha)|\beta-1|^p}{(\beta-\frac{p-1-a}{p})|\alpha|^p + (\frac{p-1-a}{p}-\alpha)|\beta|^p}.
\end{equation}
The following theorem follows from the results obtained by the author  in \cite{MR3558516}.

\begin{theorem}%[\cite{MR3558516}]
 \label{thm:C}
  For $1\leq p < \infty$  and $a < p-1$, 
 \begin{equation*}
  C_{p,a}^p= \sup\Big\{ k_{p,a}(\alpha, \beta)    : \ \alpha < \frac{p-1-a}{p} < \beta\Big\}.
 \end{equation*}
\end{theorem}

Using this result one can check, that for $p>1$ and $a=0$,
\begin{equation}
\label{eq:Cp-intr}
C_{p,0} ^p = \sup_{\alpha\leq (p-1)/p} k_{p,0}(\alpha, 1) =  %\frac{|\alpha-1|^p}{p(1-\alpha)-1 + |\alpha|^p} = 
  \begin{cases}
  \frac{1}{(p-1)^p} & \text{if } 1<p\leq 2,\\
  \frac{(1+|\alpha_p|)^{p-2}}{p-1} & \text{if } p> 2,
  \end{cases}
\end{equation}
where, for $2<p<\infty$,  $\alpha_p\in\RR$ is the unique negative solution to the equation $(p-1)\alpha_p + 2-p = |\alpha_p|^{p-2}\alpha_p$. 
Also, for $p= 1$ and $a<0$,
\begin{equation}
\label{eq:C1-intr}
 C_{1,a} =  \frac{1-a}{-a}.
\end{equation}
However, in general there seems to be no simple expression for $C_{p,a}$.

It turns out that the values of the constants  $A_{p,a}$, $B_{p,a}$ are also connected to suprema of the function $k_{p,a}$ over certain sets (see Lemma~\ref{lem:lower_bound} below). Moreover, simple expressions similar to \eqref{eq:Cp-intr} or \eqref{eq:C1-intr} can be found.

We start with bounds for $H-I$ on positive decreasing functions. The results are new apart from the cases 
$1< p\leq 2$ with $a=0$ (when $A_{p,0}=B_{p,0}=C_{p,0} = 1/(p-1)$) and $p\geq 2$ with $-1< a \leq 0$ (see \cite{MR2747011}). 
In particular, in contrast to \cite{MR2747011}, we are able to obtain sharp estimates also for increasing  (power) weights.

\begin{theorem}
\label{thm:A}
 For $1\leq p < \infty$  and $-1 < a < p-1$, 
 \begin{equation*}
A_{p,a}^p=
\begin{cases}
 \frac{1+a}{-a} & \text{ if } p=1,\\
  \bigl(\frac{1+a}{p-1-a}\bigr)^p & \text{ if } 1<p\leq2,\quad p-1-p^{\frac{p-2}{p-1}} \leq a < p-1,\\
  k_{p,a}(\alphaA,1) & \text{ if } 1<p\leq2,\quad p-2 <a < p-1-p^{\frac{p-2}{p-1}},\\
  \frac{1+a}{p-1-a} & \text{ if } 1<p\leq 2,\quad -1 <a\leq p-2,\\
\bigl(\frac{1+a}{p-1-a}\bigr)^p & \text{ if } p>2 ,\quad p^{\frac{p-2}{p-1}} -1 \leq a < p-1,\\
   k_{p,a}(0,\betaA) & \text{ if } p>2 ,\quad 0 <a < p^{\frac{p-2}{p-1}}-1,\\
  \frac{1+a}{p-1-a} & \text{ if } p>2,\quad -1 <a\leq 0.\\
  \end{cases}
\end{equation*}
Here $\alphaA$ and $\betaA$ are defined as follows: for $1<p<2$ and $p-2 <a < p-1-p^{\frac{p-2}{p-1}}$, $\alphaA=\alphaA(p,a)$ is the unique solution of the equation
\begin{equation*}
-(p-1)(1-\alphaA) + 1+a -(1+a)\alphaA^{p-1} =0
\end{equation*}
in the interval $(0, (p-1-a)/p)$;
 for $p>2$ and $0 <a < p^{\frac{p-2}{p-1}}-1$, $\betaA=\betaA(p,a)$ is the unique solution of the equation
\begin{equation*}
-(p-1-a)(1-\betaA)^{p-1}-(p-1)\betaA +p-1-a =0
\end{equation*}
in the interval $((p-1-a)/p,1)$.
\end{theorem}

We turn to estimates of $H-I$ on positive functions. The results are new for $a\neq 0$ (see \cite{MR3868629}).

\begin{theorem}
\label{thm:B}
For $1\leq p < \infty$  and $a < p-1$, 
 \begin{equation*}
B_{p,a}^p=
\begin{cases}
\frac{1-a}{-a} & \text{ if } p=1,\\
C_{p,a}^p & \text{ if } 1<p <2,\quad p-2 <a < p-1,\\
k_{p,a}(0,\betaB) & \text{ if } 1<p < 2,\quad a\leq p-2,\\
\max\{A_{p,a}\indbr{-1<a}, 1 \} & \text{ if } p\geq 2,\quad a< p-1.
\end{cases}
\end{equation*}
Here $\betaB$ is defined as follows: for $1<p<2$ and $a \leq p-2$, $\betaB=\betaB(p,a)$ is the unique solution of the equation
\begin{equation*}
-(p-1-a)|1-\betaB|^{p-2}(\betaB-1)-(p-1)\betaB +p-1-a =0
\end{equation*}
in the interval $(1,\infty)$.
\end{theorem}

Note that for $p=2$ we simply have  $B_{2,a} = A_{2,a} = (\frac{1+a}{1-a})^2 > 1$ if
 $1>a >0$ and $B_{2,a} =  1$ if $a \leq 0$. 
Similarly, one can pinpoint when $B_{p,a} = A_{p,a}$ and when $B_{p,a} =  1$ for $p>2$, see Remark~\ref{rem:transition} below.
Let us also remark that for   $1<p<2$ and  $a$  sufficiently close to $p-1$ we have in fact 
\[
 B_{p,a}^p=C_{p,a}^p= \Bigl(\frac{1+a}{p-1-a}\Bigr)^p,
 \]
while for $a< p-2$, $B_{p,a}<C_{p,a}$.

We conclude with some comments about sharp comparison of $L^p$-norms of the Hardy operator and its dual. 
The following corollary to Theorem~\ref{thm:A} extends the results obtained by Kolyada \cite{MR3180926} to the case of power weights.

\begin{corollary}
\label{cor:E}
 For $1\leq p < \infty$, $-1< a< p-1$, and any nonnegative $f\colon [0,\infty)\to [0,\infty)$,
  \begin{equation}
  \label{eq:cor:E}
  A_{p,a}^{-p} \int_0^\infty |Hf(t)|^p t^a dt 
  \leq \int_0^\infty |H^* f(t)|^p t^a dt 
  \leq A_{p,p-2-a}^p \int_0^\infty |Hf(t)|^p t^a dt.
  \end{equation}
The constants are best possible.
\end{corollary}

In general, such a two-sided comparison does not hold for not necessarily nonnegative functions and all $1\leq p<\infty$, $a<p-1$, but we have the following sharp inequality. Below it is natural to assume that
\begin{equation*}
 H\bigl(|f|\bigr)(1) < \infty , \quad  H^*\bigl(|f|\bigr)(1) < \infty,
\end{equation*}
so that $Hf$ and $H^*f$ are well defined.

\begin{proposition}
\label{prop:F}
 For $1\leq p < \infty$, $a< p-1$, 
 and any $f\colon [0,\infty)\to \RR$,
  \begin{equation}
    \label{eq:prop:F}
  \int_0^\infty |Hf(t)|^p t^a dt 
  \leq C_{p,a}^{p}  \int_0^\infty |H^* f(t)|^p t^a dt.
  \end{equation}
The constant is best possible.
\end{proposition}

We remark that the observation that for $a=0$ the best constant in the inequality~\eqref{eq:prop:F} is less or equal than $C_{p,0}$ was made by Boza and Soria \cite{MR3868629} and that their argument extends verbatim to the setting of power weights; our contribution is showing that this is in fact best possible.

Finally, let us mention a related open problem, posed in~\cite{MR3868629}, concerning the value of the best constant $D_{p,a}$ such that the inequality
\begin{equation}
  \label{eq:def-D}
 \int_0^\infty |Hf(t)|^p t^a dt \leq D_{p,a}^{p} \int_0^\infty |H^* f(t)|^p t^a dt
  \end{equation}
holds for all positive and  decreasing functions $f\colon [0,\infty)\to\RR$. It is known \cite{MR3868629} that $D_{p,0} = \frac{1}{p-1}$ for $1<p\leq 2$ and 
\[
 D_{p,0} = \Bigl(\frac{p}{(p-1) \Gamma(p+1)} \Bigr)^{1/p}
\]
for \emph{integer} $p\geq 2$, while  for general $p\geq 2$,
\[
 \Bigl(\frac{p}{(p-1) \Gamma(p+1)} \Bigr)^{1/p}\leq D_{p,0} \leq \min\Bigl\{ \frac{e}{p-1}, \frac{1}{(p-1)^{1/p}}\Bigr\}
\]
(those inequalities are asymptotically sharp).  

The organization of the rest of the paper is the following. In Section~\ref{sec:prelim} we provide some preliminary results and describe the method of the proof of Theorems~\ref{thm:C}, \ref{thm:A}, and \ref{thm:B}. The proofs of Theorems  \ref{thm:A} and \ref{thm:B} for $p>2$ and $a=0$ are also presented therein.

The proofs of Theorems~\ref{thm:C}, \ref{thm:A}, and \ref{thm:B} (and Remark~\ref{rem:transition}) in their full generality  are presented in Sections~\ref{sec:C}, \ref{sec:A},  \ref{sec:B} respectively.

In Section~\ref{sec:Hstar} we prove Corollary~\ref{cor:E} and Proposition~\ref{prop:F}.

\section{Preliminaries}
\label{sec:prelim}

\subsection{Lower bounds for best constants}

We first adapt the construction of the extremal family from \cite{MR3558516} to our setting in order to get lower bounds for $B_{p,a}$ and $A_{p,a}$.
Recall that $k_{p,a}(\alpha,\beta)$ is defined by~\eqref{eq:def_k}.  

\begin{lemma}
\label{lem:lower_bound}
For $1\leq p < \infty$  and $a< p-1$,
 \begin{align*}
C_{p,a}^p &\geq \sup\Big\{ k_{p,a}(\alpha, \beta)    : \ \alpha < \frac{p-1-a}{p} < \beta\Big\},\\
B_{p,a}^p &\geq \sup\Big\{ k_{p,a}(\alpha, \beta) : \ 0\leq \alpha < \frac{p-1-a}{p} < \beta\Big\},\\
A_{p,a}^p &\geq \sup\Big\{ k_{p,a}(\alpha, \beta) : \ 0\leq \alpha < \frac{p-1-a}{p} < \beta\leq 1\Big\}.
\end{align*}
\end{lemma}

Note that in  case of $A_{p,a}$ the assertion is trivial unless $a>-1$. Later on we shall verify that there are equalities in Lemma~\ref{lem:lower_bound} and identify more explicit expressions for the constants.

\begin{proof}[Proof of Lemma~\ref{lem:lower_bound}]
Fix $1\leq p<\infty$ and $a<p-1$. %$m>-2(p-1)/p$.
If $\alpha<\frac{p-1-a}{p}<\beta$, then the function
\begin{equation}
\label{eq:f-ab}
f(t)= f_{\alpha,\beta}(t) = \beta t^{\beta- 1} \indbr{t\in[0,1)}+\alpha t^{\alpha-1}\indbr{t\in[1,\infty)}, \quad  t\geq 0,
\end{equation}
 clearly satisfies $\int_0^\infty |f(t)|^p t^a dt <\infty$. Moreover,
\begin{equation*}
H f(t) = t^{\beta-1} \indbr{t\in[0,1)}+ t^{\alpha-1}\indbr{t\in[1,\infty)}
\end{equation*}
and
\begin{align*}
\int_0^\infty |Hf(t) - f(t)|^p t^a dt &= \frac{|\beta-1|^p}{p(\beta-1)+a+1} -\frac{|\alpha-1|^p}{p(\alpha-1)+a+1},\\
\int_0^\infty |f(t)|^p t^a dt &= \frac{|\beta|^p}{p(\beta-1)+a+1} -\frac{|\alpha|^p}{p(\alpha-1)+a+1}.
\end{align*}
Considering the ratio of those two quantities and taking the supremum over all  $\alpha<\frac{p-1-a}{p}<\beta$ yields the lower bound for $C_{p,a}^p $. 
In order to estimate $B_{p,a}^p$ (resp.\ $A_{p,a}^p$) we have to restrict the admissible parameters $\alpha$, $\beta$ to nonnegative numbers (resp.\ numbers in $[0,1]$), in order for the function $f$ to be positive (resp.\ positive and decreasing).
\end{proof}

\subsection{Method of the proof}
It is more demanding to find good upper bounds for the constants $C_{p,a}$, $B_{p,a}$, and $A_{p,a}$.

First of all, we need the following classical lemma (which combined with H\"older's inequality can be used to prove Hardy's inequality~\eqref{eq:Hardy}).
  
\begin{lemma} 
\label{lem:int_by_parts}
Suppose that $f\colon [0,\infty)\to\RR$ satisfies $\int_0^\infty |f(t)|^p t^a dt<\infty$ for some $1\leq p< \infty$ and $a<p-1$. Then
\begin{equation}
 \label{eq:int_by_parts}
 \int_0^\infty |Hf(t)|^p t^a dt=\frac{p}{p-1-a} \int_0^\infty |Hf(t)|^{p-2} Hf(t)f(t)  t^a dt
\end{equation}
(and both integrals are finite).  
\end{lemma}

For the reader's convenience we provide the proof from \cite[Chapter 1]{MR1963498}.

\begin{proof}[Proof of Lemma~\ref{lem:int_by_parts}]
We may suppose that $f\not\equiv 0$. Denote $F(t) = \int_0^t f(s) ds $. 
 By H\"older's inequality (with obvious changes for $p=1$),
\[
 \int_0^t |f(s)|  ds 
 \leq
 \Bigl(  \int_0^t |f(s)|^p s^{a} ds \Bigr)^{1/p}
 \Bigl(  \int_0^t   s^{-a/(p-1)} ds \Bigr)^{(p-1)/p}.
\]
Thus, $F(t)  = o(t^{(p-1-a)/p})$ for $t\to 0^+$.

Fix any $\varepsilon>0$.
Applying the same argument as above on the interval $[\xi,t]$ with $\xi=\xi(\varepsilon)$ large enough, we see that
\[
 F(t) - F(\xi) \leq \varepsilon t^{(p-1-a)/p}.
\]
As the exponent on the right hand side is positive, we conclude that for $t$ large enough, $F(t)\leq 2 \varepsilon t^{(p-1-a)/p}$. Hence, $F(t)  = o(t^{(p-1-a)/p})$ for $t\to \infty$. 
 
Therefore, integration by parts yields
\begin{align*}
  \int_0^\infty |Hf(t)|^p t^a dt
  &=   \int_0^\infty |F(t)|^p t^{a-p} dt \\
   &= \Bigl[-\frac{1}{p-1-a} |F(t)|^p t^{a+1-p} \Bigr]^\infty_0  \\
   &\qquad + \frac{p}{p-1-a} \int_0^\infty  |F(t)|^{p-2} F(t) f(t) t^{1+a-p} dt\\
  &= 
  \frac{p}{p-1-a} \int_0^\infty |Hf(t)|^{p-2} Hf(t)f(t)  t^a dt.
\end{align*}
(To see that both integrals are indeed finite one can first integrate by parts on a finite interval and use H\"older's inequality).
\end{proof}

The next proposition reduces the task of proving that an inequality holds for all functions from a given class to constructing a majorant of adequate form on an appropriate domain.

\begin{proposition} 
\label{prop:majorize}
 Suppose that for some $1\leq p< \infty$, $a<p-1$ and some constants $K\in(0,\infty)$, $D\in\RR$
 the inequality
 \begin{equation*}
  |1-x|^p - K^p |x|^p \leq D \Bigl( \frac{p-1-a}{p} - x\Bigr)
 \end{equation*}
holds for all $x\in\RR$ (resp.\ for all $x\geq 0$, resp.\ for all $x\in[0,1]$).
Then the inequality
\begin{equation*}
\int_0^\infty |Hf(t) -f(t)|^p t^a dt \leq K^p \int_0^\infty |f(t)|^p t^a dt.
\end{equation*}
holds for all (resp.\ for all positive, resp.\ for all positive and decreasing) functions $f\colon [0,\infty) \to\RR$. In other words, $C_{p,a}\leq K$ (resp. $B_{p,a}\leq K$, resp. $A_{p,a}\leq K$). 
\end{proposition}

\begin{proof}
Denote $\mathcal{D}\coloneqq \RR^2$ (resp.\ $\mathcal{D}\coloneqq \{(x,y)\in\RR^2 : x,y\geq 0 \} $, resp.\ $\mathcal{D}\coloneqq \{(x,y)\in\RR^2 : y\geq x \geq 0 \} $).
For $(x,y)\in \mathcal{D}$ let
\begin{align*}
 V(x,y) &\coloneqq   |y-x|^p - K^p |x|^p,\\
 U(x,y) &\coloneqq D |y|^{p-2}y\Bigl( \frac{p-1-a}{p}y - x\Bigr). 
\end{align*}
By our hypothesis and homogeneity, $V\leq U$ on $\mathcal{D}$.

Let $f\colon [0,\infty) \to\RR$  be any (resp.\ any positive, resp.\ any positive and decreasing) function such that $\int_0^\infty |f(t)|^p t^a dt < \infty$. Clearly if $f$ is positive, so is $Hf$; if $f$ is moreover decreasing, then $Hf\geq f$. Thus, $(f(t),Hf(t))\in \mathcal{D}$ and  by Lemma~\ref{lem:int_by_parts},
\begin{equation*}
 \int_0^\infty V(f(t),Hf(t)) t^a dt \leq  \int_0^\infty U(f(t),Hf(t)) t^a dt = 0. 
\end{equation*}
This ends the proof.
\end{proof}

\subsection{Warm-up: the unweighted setting}
Nearly all proofs below are based on the scheme outlined in Proposition~\ref{prop:majorize}. Some calculations or educated guesses are needed to find the right candidate for the optimal constant $K$, but apart from that the proofs involve only elementary calculations. Nonetheless, the consideration of a number of different inequalities and the presence of two parameters, $p$ and $a$, forces us to split the reasoning into several cases and yields some complications. 
Therefore, we believe it is instructive to present two proofs from the unweighted setting in this place. 

For $a=0$ and $1<p\leq 2$ we have $C_{p,0} = (p-1)^{-1}$, see~\cite{MR2595549}. Since the extremal functions can be chosen to be positive and decreasing we in fact have   $A_{p,0} = B_{p,0} = C_{p,0} = (p-1)^{-1}$. 
For $p>2$ the situation is more interesting, as $A_{p,0} = (p-1)^{-1/p}  < B_{p,0} = 1 < C_{p,0}$, see~\cite{MR2747011,MR3868629,MR3558516}. 
Below we explain how one can use Proposition~\ref{prop:majorize} to guess the values of $A_{p,0}$ and  $B_{p,0}$ to begin with and then confirm this guess.

Suppose that we do not yet know the value of $A_{p,0}$. 
In order to use  Proposition~\ref{prop:majorize} we need to find a candidate $K=K(p)$ for $A_{p,0}$ and some constant $D=D(K,p)$ such that
\begin{align*}
 v(x) &= |1-x|^p -  K^p |x|^p,\\
 u(x) &= D \bigl(\frac{p-1}{p} - x\bigr)
\end{align*}
satisfy $v(x)\leq u(x)$ for $x\in[0,1]$. In particular, the conditions  $v(0)\leq (0)$ and $v(1)\leq u(1)$ yield $D\geq p/(p-1)$ and $K^p\geq D/p$, respectively.
This implies $K \geq (p-1)^{-1/p}$.

If on the other hand we  want to find the value of $B_{p,0}$, we need to assure that, for an appropriate, possibly different choice of $K$ and $D$, $v(x)\leq u(x)$ for all $x\geq 0$. In particular, the conditions  $v(0)\leq (0)$ and $\lim_{x\to\infty } v(x)/x^p\leq \lim_{x\to\infty } u(x)/x^p$ yield $D\geq p/(p-1)$ and $K\geq 1$, respectively.

It turns out that these lower bounds---which are imposed solely by our method of proof---are also the correct upper bounds.

\begin{proof}[Proof of Theorem~\ref{thm:A} for $p>2$ and $a=0$]
For $x\in[0,1]$ define
\begin{align*}
 A &=  \frac{1}{(p-1)^{1/p}},\\
 v(x) &= |1-x|^p -  \frac{1}{p-1} |x|^p,\\
 u(x) &= \frac{p}{p-1} \bigl(\frac{p-1}{p} - x\bigr).
\end{align*}
Observe first that $A<1$, $u$ is the tangent to $v$ at $x=0$, and $v(1)=u(1)$.
Moreover,  a direct calculation of $v''$ shows that $v$ is convex on $[0,1/(1+A^{p/(p-2)})]$ and concave on $[1/(1+A^{p/(p-2)}),1]$ (cf.\ Lemma~\ref{lem:ccc} below). It follows from those properties that $v\leq u$ on $[0,1]$.
 Thus, by Proposition~\ref{prop:majorize},  $A_{p,0}\leq A$. Since by Lemma~\ref{lem:lower_bound} also $A_{p,0}\geq k_{p,0}(0,1) ^{1/p} = A$, we conclude that  $A_{p,0}=A$. This ends the proof.
\end{proof}

\begin{proof}[Proof of Theorem~\ref{thm:B} for $p>2$ and $a=0$]
This time let us define
\begin{align*}
 v(x) &= |1-x|^p - |x|^p,\\
 u(x) &= \frac{p}{p-1} \bigl(\frac{p-1}{p} - x\bigr)
\end{align*}
for $x \geq 0$.
Clearly, $v$ is decreasing on $[0,\infty)$.

Moreover, $v$ convex on $[0,1/2]$,  $v(0)=u(0)$, and $v(1/2)=0 < u(1/2)$. Hence $v\leq u$ on $[0,1/2]$.

Since on $(1/2,(p-1)/p)$, $v$ is negative and $u$ positive, we have $v\leq u$ on $[1/2,(p-1)/p]$. Let us  for now the following claim for granted:
\begin{equation}
 \label{eq:B_p-zero:claim}
 v'((p-1)/p) \leq u'((p-1)/p).
\end{equation}
Since on the interval $[1/2,\infty)$ the function $v$ is concave,  \eqref{eq:B_p-zero:claim} implies that $v\leq u$ also on $[(p-1)/p,\infty)$.
Thus, by Proposition~\ref{prop:majorize},  $B_{p,0}\leq 1$. Since, by Lemma~\ref{lem:lower_bound},
\[
 B_{p,0}\geq \lim_{\beta \to \infty }k_{p,0}(0,\beta) ^{1/p} = 1,
\]
we conclude that $B_{p,0}=1$. 
It remains to check the claim \eqref{eq:B_p-zero:claim}, but apart from that the proof is finished.
\end{proof}

Even though  it is not hard to convince oneself, e.g., numerically that \eqref{eq:B_p-zero:claim} holds true, proving this is somewhat cumbersome. We present one possible approach.  

\begin{proof}[Proof of the claim]
 The inequality \eqref{eq:B_p-zero:claim} is equivalent to 
 \begin{equation}
  \label{eq:B_p-zero:claim:equiv}
  p^{p-1}\leq p-1 + (p-1)^p.
 \end{equation}
By the inequality between the weighted arithmetic mean and the weighted geometric mean,
\[
 p-1 + (p-1)^p \geq \frac{1}{\theta^\theta(1-\theta)^{1-\theta}}(p-1)^{1-\theta + \theta p}, \quad \theta\in(0,1).
\]
Using this with $\theta \in \{ \frac{1}{2}, \frac{3}{4},  1^-\}$ we estimate $p-1 + (p-1)^p$ from below and reduce the task of  proving \eqref{eq:B_p-zero:claim:equiv} to  checking that
\begin{align}
\label{eq:amgm-1}
   2(p-1)^{(p+1)/2} &\geq p^{p-1} \quad \text{ for } p\in[2,2.5],\\
   \label{eq:am-gm-2}
\frac{4}{3^{3/4}}  (p-1)^{(3p+1)/4} & \geq p^{p-1}  \quad \text{ for } p\in[2.5,3.5],\\
\label{eq:am-gm-3}
 (p-1)^p  & \geq p^{p-1} \quad \text{ for } p\in[3.5,\infty).
\end{align}

Each of those inequalities holds at the endpoints. Moreover, in each case the difference of the logarithms of both sides is a concave function, since
\begin{align*}
  \Bigl( \frac{p+1}{2} \ln(p-1) - (p-1)\ln(p) \Bigr)''  =  \frac{-(2 - 2 p + p^2 + p^3)}{2 (p-1)^2 p^2}      &  \leq 0\enspace \text{ for } p\in[2,2.5],\\
    \Bigl( \frac{3p+1}{4} \ln(p-1) - (p-1)\ln(p) \Bigr)''  =  \frac{-(4 - 4 p + 3 p^2 + p^3)}{4 (p-1)^2 p^2}   &  \leq 0\enspace \text{ for } p\in[2.5,3.5],\\
  \bigl( p \ln(p-1) - (p-1)\ln(p) \bigr)''  =   \frac{(-1 + p - p^2)}{(p-1 )^2 p^2}     &        \leq 0\enspace \text{ for } p\in[3.5,\infty).
\end{align*}
Thus \eqref{eq:amgm-1}, \eqref{eq:am-gm-2}, \eqref{eq:am-gm-3} indeed hold true.
\end{proof}

 \subsection{Auxiliary lemmas}
Below are some easy technical results to be used in the proofs. 

\begin{lemma}
  \label{lem:ccc}
  For $K>0$, $p>1$, and $x\in\RR$, denote
  \[
   v(x) =  |x-1|^p-K^p|x|^p.
  \]
Then following holds. 
\begin{enumerate}[a)]
 \item For $1<p<2$ and $K>1$, $v$ is concave, convex, concave, on the intervals $(-\infty,\frac{1}{1+K^{p/(p-2)}})$, $(\frac{1}{1+K^{p/(p-2)}},\frac{1}{1-K^{p/(p-2)}})$, $(\frac{1}{1-K^{p/(p-2)}},\infty)$, respectively. 
  \item For $1<p<2$ and $0<K<1$, $v$ is convex, concave, convex on the intervals $(-\infty,\frac{1}{1-K^{p/(p-2)}})$, $(\frac{1}{1-K^{p/(p-2)}},\frac{1}{1+K^{p/(p-2)}})$, $(\frac{1}{1+K^{p/(p-2)}},\infty)$, respectively. 
   \item For $p>2$ and $K>1$, $v$ is concave, convex, concave, on the intervals $(-\infty,\frac{1}{1-K^{p/(p-2)}})$, $(\frac{1}{1-K^{p/(p-2)}},\frac{1}{1+K^{p/(p-2)}})$, $(\frac{1}{1+K^{p/(p-2)}},\infty)$, respectively. 
   \item For $p>2$ and $0<K<1$, $v$ is convex, concave, convex on the intervals $(-\infty,\frac{1}{1+K^{p/(p-2)}})$, $(\frac{1}{1+K^{p/(p-2)}},\frac{1}{1-K^{p/(p-2)}})$, $(\frac{1}{1-K^{p/(p-2)}},\infty)$, respectively. 
\end{enumerate}
\end{lemma}
 
 \begin{proof}
  For $x\neq \{0,1\}$ we have
  \[
   v''(x) = p(p-1)\Bigl(|x-1|^{p-2} - K^p |x|^{p-2}\Bigr),
  \]
which is positive if and only if $|1-1/x|^{p-2}\geq K^{p}$. We leave the rest of the details to the reader.
 \end{proof}

 \begin{lemma}\label{lem:rolle}
Suppose that $v:\RR\to\RR$ is continuously differentiable and strictly concave on $(-\infty,a)$, strictly convex on $(a,b)$, and strictly concave on $(b,\infty)$ for some  $a,b\in\RR$. Let $u:\RR\to\RR$ be an affine function tangent to $v$ at two points. Then $v(x)\leq u(x)$ for $x\in\RR$.
\end{lemma}

\begin{proof} See Lemma~3.6 in \cite{MR3558516}
\end{proof}

 \begin{lemma}
 \label{lem:ppp}
 For $1<p<2$ we have $(p-1)^{p-1} < p^{p-2}$. For $p>2$ the reverse inequality holds true.
\end{lemma}
 
 \begin{proof}
  We have equality in the limit for $p\to 1^+$ or $p\to 2^-$ and the (reverse) inequality holds for $p\to \infty$.  
 Since
\begin{equation*}
\big( (p-2)\ln (p)- (p-1)\ln (p-1)\big)'' = \frac{p-2}{(p-1)p^2},
\end{equation*}
the difference of the logarithms of both sides is a concave function on $(1,2)$ and a convex one on $(2,\infty)$. This yields the assertion.
 \end{proof}

 \section{Estimates for \texorpdfstring{$H-I$}{H--I} on general functions}
\label{sec:C}

 For the sake of completeness let us show how to deduce Theorem~\ref{thm:C} from the results of \cite{MR3558516}.
 Apart from notational changes the proof in \cite{MR3558516} follows the scheme outlined above, in Proposition~\ref{prop:majorize}, and we refer to that article for the construction of the special functions $v$, $u$ and related technical details.

\begin{proof}[Proof of Theorem~\ref{thm:C} for $p>1$]
Fix $1< p<\infty$ and $a<p-1$. Let $f\colon [0,\infty)\to \RR$ be such that  $\int_0^\infty |f(t)|^p t^a dt<\infty$.  
If we substitute $g(t) = f(t) t^{a/p}$, then the inequality
\[
 \int_0^\infty |Hf(t)-f(t)|^p t^a dt\leq C_{p,a}^p \int_0^\infty |f(t)|^p t^a dt
\]
transforms   into
\[
  \int_0^\infty \Bigl|\frac{1}{t^{1-a/p}}\int_0^t g(s) s^{-a/p}ds -g(t)\bigr|^p dt\leq C_{p,a}^p \int_0^\infty |g(t)|^p dt.
\]
Thus the assertion follows from Theorem~1.1 in \cite{MR3558516} (applied with $ m = -2a/p >-2(p-1)/p$ and $\lambda=1$). 
\end{proof}

While formally the case $p=1$ is excluded in the formulation in \cite{MR3558516}, this is only because therein the focus was on the case $m=0$ (for which $p>1$ is needed). The statement still holds for $p=1$ (and $m>0$ or, in the present notation, $a<0$) and actually the obvious changes one needs to introduce in the proof simplify it.

\begin{proof}[Proof of Theorem~\ref{thm:C} for $p=1$]
Fix $a <0$.
For $x\in\RR$, define
\begin{align*}
 C &=  \frac{1-a}{-a},\\
 v(x) &= |1-x| -  C |x|,\\
 u(x) &= \frac{1}{a}(a + x).
\end{align*}
We clearly have $v(0)=u(0)$ and $v(1)\leq u(1)$. 
Moreover, $v$ is a piecewise affine function and the comparison of slopes of $v$ and $u$ yields $v\leq u$ on $\RR$.
Thus, by Proposition~\ref{prop:majorize},  $C_{1,a}\leq C$. 
Since by Lemma~\ref{lem:lower_bound} also $C_{1,a}\geq \lim_{\beta\to\infty} k_{1,a}(0,\beta)  = (1-a)/(-a) $, we conclude that  $C_{1,a}=C$. This ends the proof.
\end{proof}

\section{Estimates for \texorpdfstring{$H-I$}{H--I} on decreasing functions}
\label{sec:A}

We split the proof of Theorem~\ref{thm:A} into several cases. We start with $p\in (1,2)$ (the boundary cases $p=1$, $p=2$ are simpler and we treat them separately).

\begin{proof}[\nextcase: $1<p<2$ and $p-1-p^{\frac{p-2}{p-1}} \leq a < p-1$]
For $x\in[0,1]$, define
\begin{align*}
 A &=  \frac{1+a}{p-1-a},\\
 v(x) &= |1-x|^p -  A^p |x|^p,\\
 u(x) &= \frac{(1+a)^{p-1}}{p^{p-3}(p-1-a)} \bigl(\frac{p-1-a}{p} - x\bigr).
\end{align*}
Observe first that $A>1$ (since by our assumption $a>p-1-p^{\frac{p-2}{p-1}}\geq p/2-1$) and $u$ is the tangent to $v$ at $x=(p-1-a)/p$.

Moreover, 
$v$ is concave on $[0,1/(1+A^{p/(p-2)})]$ and convex on $[1/(1+A^{p/(p-2)}),1]$. Also, $A>1$ and $p/(p-2)\leq 0$ yield
\begin{equation*}
  \frac{p-1-a}{p} = \frac{1}{1+A}\leq \frac{1}{1+A^{p/(p-2)}},
\end{equation*}
while the assumption $p-1-p^{\frac{p-2}{p-1}} \leq a$ implies that
\begin{equation*}
 v(1) = -\frac{(1+a)^p}{(p-1-a)^p}\leq - \frac{(1+a)^p}{p^{p-2}(p-1-a)} = u(1).
\end{equation*}
It follows from those properties that $v\leq u$ on $[0,1]$.
Thus, by Proposition~\ref{prop:majorize},  $A_{p,a}\leq A$. Since by Lemma~\ref{lem:lower_bound} also $A_{p,a}\geq k_{p,a}((p-1-a)/p,1) ^{1/p} = A$, we conclude that  $A_{p,a}=A$. This ends the proof.
\end{proof}

Before considering the next cases, let us present some heuristics.
It is clear from the above reasoning, that for $a < p-1-p^{(p-2)/(p-1)}$ we have to define $A$, $v$, and $u$ differently in order for the proof to work, since otherwise the majorization $ v(1) \leq u(1)$ fails.
This suggests that we should increase $A$.
At the same time we still want $u$ to have the form required by Proposition~\ref{prop:majorize} and it reasonable to guess that $u$ should be a tangent to $v$ at some point (which changes and wanders more to the left).
Moreover, it should be intuitively clear that we should increase $A$ (while  changing $u$ accordingly), until we have $v(1) = u(1)$. 

These intuitions stand behind the formal calculations below. Note that $\alphaA$ is the point at which $u$ will be tangent to $v$.

\begin{lemma}
 \label{lem:tech-A-prep}
Suppose that $1<p<2$ and $-1  < a < p-1-p^{\frac{p-2}{p-1}}$. Then
\[
 (1+a)\alpha^p + p - 1 - a - p\alpha > 0 
\]
for $\alpha\in [0,1)$.
\end{lemma}

\begin{proof}
 By Lemma~\ref{lem:ppp},  $a<p-1-p^{\frac{p-2}{p-1}}<0$. Hence,
\[
\frac{d}{d\alpha} (  (1+a)\alpha^p + p - 1 - a - p\alpha ) = p(1+a)\alpha^{p-1} - p < p(1+a) - p \leq 0,
\]
where we also used the constraints $a>-1$ and $\alpha<1$. Thus the investigated function is decreasing and since it vanishes for $\alpha=0$, the assertion holds.  
\end{proof}

Recall that the function
\[
 k_{p,a}(\alpha,1) = 
 \frac{(1-\frac{p-1-a}{p})|\alpha-1|^p }{(1-\frac{p-1-a}{p})|\alpha|^p + (\frac{p-1-a}{p}-\alpha)}
 = \frac{(1+a)|\alpha-1|^p }{(1+a)|\alpha|^p + p-1-a -p\alpha}
\]
was defined in~\eqref{eq:def_k} for $\alpha  < (p-1-a)/p$. Lemma~\ref{lem:tech-A-prep} implies that the same definition makes sense for all $\alpha\in[0,1)$ and justifies the notation  used below.

\begin{lemma}
 \label{lem:tech-A}
Suppose that $1<p<2$ and $-1  < a < p-1-p^{\frac{p-2}{p-1}}$. Then there exists exactly one number $\alphaA=\alphaA(p,a)\in[0,1)$
such that
\[
 \sup \{ k_{p,a}(\alpha,1) : 0\leq \alpha < 1\}
 = k_{p,a}(\alphaA,1).
\]
Moreover, for $p-2 < a < p-1-p^{\frac{p-2}{p-1}}$,
$\alphaA$ is the unique solution of the equation
\begin{equation*}
-(p-1)(1-\alphaA) + 1+a -(1+a)\alphaA^{p-1} =0
\end{equation*}
in the interval $(0, (p-1-a)/p)$,
whereas for $ -1< a \leq p-2$, $\alphaA=0$.
\end{lemma}

\begin{proof}
For $\alpha\in(0,1)$ the derivative $\frac{d}{d\alpha} k_{p,a}(\alpha,1)$ is of the same sign as 
\begin{multline*}
-p(1+a)(1-\alpha)^{p-1} \cdot  \bigl( (1+a)\alpha^p + p-1-a-p\alpha\bigr) 
-(1+a)(1-\alpha)^p \cdot \bigl( p(1+a)\alpha^{p-1} -p\bigr)\\
% % % &=p(1+a)(1-\alpha)^{p-1}  \Bigl( - (1+a)\alpha^p - p+1+a+p\alpha 
% % % -(1-\alpha)  \bigl( (1+a)\alpha^{p-1} -1\bigr)\Bigr)\\
=p(1+a)(1-\alpha)^{p-1}  \Bigl( - (1+a)\alpha^{p-1} + 1+a +(p-1)(\alpha-1)\Bigr).
\end{multline*}
In order to maximize $k_{p,a}(\cdot,1)$ we investigate the function 
\[
 h(\alpha) =  - (1+a)\alpha^{p-1} + 1+a +(p-1)(\alpha-1), \quad  \alpha\in[0,1].
\]
We have $h(0) = a- (p-2)$, $h(1) = 0$, and, for $\alpha\in(0,1)$, 
\begin{align*}
   h'(1^-) &=  - (p-1)a >0,\\
   h''(\alpha) &=  (p-1)(2-p)(1+a)\alpha^{p-3} > 0.
   \end{align*}
 Thus $h$ is strictly convex on $[0,1]$.
 
 If $p-2 < a < p-1-p^{\frac{p-2}{p-1}}$, then $h$ has a unique zero $\alphaA = \alphaA(p,a)$ in $[0,1]$. Moreover, $h((p-1-a)/p) < 0$ (this turns out to be equivalent to $a < p-1-p^{\frac{p-2}{p-1}}$), so $\alphaA < (p-1-a)/p$. The maximum of $k_{p,a}(\cdot,1)$ on $[0,1]$ is attained  at $\alphaA$.
 
 If $ -1 < a \leq p-2$, then $h$ is non-positive on  $[0,1]$ and thus the maximum of $k_{p,a}(\cdot,1)$ on $[0,1]$ is attained  at $\alphaA=0$. This ends the proof. 
\end{proof}

We continue the proof of Theorem~\ref{thm:A}.

\begin{proof}[\nextcase:  $1<p<2$ and $p-2 <a < p-1-p^{\frac{p-2}{p-1}}$]
Let $\alphaA=\alphaA(p,a)$ be the constant from the preceding  lemma.
 This time, for $x\in[0,1]$, define
\begin{align*}
 A^p &=  k_{p,a}(\alphaA,1) = \frac{(1+a) (1-\alphaA)^p}{p-1-a - p \alphaA + (1+a)\alphaA^p},\\
 v(x) &= |1-x|^p -  A^p |x|^p,\\
 u(x) &= \frac{pA^p}{1+a} \bigl(\frac{p-1-a}{p} - x\bigr).
\end{align*}

By Lemma~\ref{lem:tech-A}, 
\[
 A^p \geq \frac{(1+a) (1-x)^p}{p-1-a - p x+ (1+a)x^p}
\]
for $x\in[0,1)$. This is equivalent to $v\leq u$ on $[0,1)$ (and the inequality at $x=1$ is obvious). Thus, by Proposition~\ref{prop:majorize},  $A_{p,a}\leq A$. Since, by Lemma~\ref{lem:lower_bound}, also $A_{p,a}^p\geq k_{p,a}(\alphaA,1) = A^p$, we conclude that  $A_{p,a}=A$. This ends the proof.
\end{proof}

\begin{proof}[\nextcase:  $1<p<2$ and $-1 <a\leq p-2$]
For $x\in[0,1]$ define
\begin{align*}
 A^p &=  \frac{1+a }{p-1-a},\\
 v(x) &= |1-x|^p -  A^p |x|^p,\\
 u(x) &= \frac{p}{p-1-a} \bigl(\frac{p-1-a}{p} - x\bigr).
\end{align*}

By Lemma~\ref{lem:tech-A}, 
\[
 A^p \geq \frac{(1+a) (1-x)^p}{p-1-a - p x+ (1+a)x^p}
\]
for $x\in[0,1)$. This is equivalent to $v\leq u$ on $[0,1)$ (and the inequality at $x=1$ is obvious). Thus, by Proposition~\ref{prop:majorize},  $A_{p,a}\leq A$. Since, by Lemma~\ref{lem:lower_bound}, also $A_{p,a}^p\geq k_{p,a}(0,1) = A^p$, we conclude that  $A_{p,a}=A$. This ends the proof.
\end{proof}

\begin{proof}[\nextcase: $p=2$]
The proof can be repeated verbatim as above, but since $p-1-p^{(p-2)/(p-1)} = p-2$, the most complicated `middle' case vanishes. 
\end{proof}

\begin{proof}[\nextcase: $p=1$] Here we can just take $A =  \frac{1+a}{-a}$, $v(x) = u(x)= 1-x + \frac{1+a}{a}x.$
\end{proof}

We move to $p>2$.  The proof is similar to the one for the range $1\leq p\leq 2$. The main difference is that 
now---provided that $A>1$---the function
\[
 x\in[0,1] \mapsto |1-x|^p - A^p |x|^p
\]
is first convex and then concave (see Lemma~\ref{lem:ccc}). This makes the majorization $v\leq u$ near zero a problem, whereas $v\leq u$  on some interval with the right end equal to $1$ will follow automatically by concavity.

\begin{proof}[\nextcaselabel{case:A-6}:  $p>2$ and $p^{\frac{p-2}{p-1}} -1 \leq a < p-1$]
For $x\in[0,1]$ define
\begin{align*}
 A &=  \frac{1+a}{p-1-a},\\
 v(x) &= |1-x|^p -  A^p |x|^p,\\
 u(x) &= \frac{(1+a)^{p-1}}{p^{p-3}(p-1-a)} \bigl(\frac{p-1-a}{p} - x\bigr).
\end{align*}
Observe first that $A>1$ (since by our assumption $a>p^{\frac{p-2}{p-1}} - 1\geq p/2-1$) and $u$ is the tangent to $v$ at $x=(p-1-a)/p$.

Moreover, by Lemma~\ref{lem:ccc}, $v$ is convex on $[0,1/(1+A^{p/(p-2)})]$ and concave on $[1/(1+A^{p/(p-2)}),1]$. Also, $A>1$ and $p/(p-2)\geq 1$ imply that
\begin{equation*}
  \frac{p-1-a}{p} = \frac{1}{1+A}\geq \frac{1}{1+A^{p/(p-2)}},
\end{equation*}
while the assumption $p^{\frac{p-2}{p-1}} - 1 \leq a$ yields
\begin{equation*}
 v(0) = 1\leq \frac{(1+a)^p}{p^{p-2}} = u(0).
\end{equation*}
It follows from those properties that $v\leq u$ on $[0,1]$.
Thus, by Proposition~\ref{prop:majorize},  $A_{p,a}\leq A$. Since by Lemma~\ref{lem:lower_bound} also $A_{p,a}\geq k_{p,a}((p-1-a)/p,1) ^{1/p} = A$, we conclude that  $A_{p,a}=A$. This ends the proof.
\end{proof}

 \begin{remark}
  \label{rem:B=A-easy}
 In the proof of Case~\ref{case:A-6} above $u$ majorizes $v$ not only on $[0,1]$, but on $[0,\infty)$. As a conclusion, by Proposition~\ref{prop:majorize}, we also get $B_{p,a} = A_{p,a}$ for $p>2$ and $p^{\frac{p-2}{p-1}} -1 \leq a < p-1$.
 \end{remark}

For $a < p^{(p-2)/(p-1)}-1$ the majorization $v(0) \leq u(0)$ (with $A$, $v$, and $u$ as in the above proof of Case~\ref{case:A-6}) fails, so the construction of $u$ has to be modified.

Obviously, the definition of function
\[
 k_{p,a}(0,\beta) = 
 \frac{\beta-\frac{p-1-a}{p} + \frac{p-1-a}{p}|\beta-1|^p}{\frac{p-1-a}{p}|\beta|^p}
 =  \frac{p\beta-(p-1-a) + (p-1-a)|\beta-1|^p}{(p-1-a)|\beta|^p}.
\]
can be extended to all $\beta>0$.

\begin{lemma}
 \label{lem:tech-A-2}
Suppose that $p>2$ and $-1  < a < p^{\frac{p-2}{p-1}}-1$. Then there exists exactly one number $\betaA=\betaA(p,a)\in[0,1)$
such that
\[
 \sup \{ k_{p,a}(0,\beta) : 0< \beta  \leq 1\}
 = k_{p,a}(0,\betaA).
\]
Moreover, for $0< a <p^{\frac{p-2}{p-1}}-1$,
$\betaA$ is the unique solution of the equation
\begin{equation*}
p-1-a - (p-1)\beta -  (p-1-a)(1-\beta)^{p-1}=0
\end{equation*}
in the interval $((p-1-a)/p,1)$,
whereas for $ -1< a \leq 0$, $\betaA=1$.
\end{lemma}

\begin{proof}
For $\beta>0$ the derivative $\frac{d}{d\beta} k_{p,a}(0,\beta)$ is of the same sign as 
\begin{multline*}
\bigl( p + (p-1-a)p|\beta-1|^{p-2}(\beta-1) \bigr)  \cdot (p-1-a)\beta^p\\
 -\bigl(p\beta-(p-1-a) + (p-1-a)|\beta-1|^p\bigr)\cdot (p-1-a)p\beta^{p-1}\\
=p(p-1-a)\beta^{p-1}  \Bigl(p-1-a - (p-1)\beta + (p-1-a)|\beta-1|^{p-2}(\beta-1) \Bigr).
\end{multline*}
In order to maximize $k_{p,a}(0,\cdot)$ we investigate the function 
\[
 h(\beta) = p-1-a - (p-1)\beta + (p-1-a)|\beta-1|^{p-2}(\beta-1), \quad  \beta\in[0,1].
\]
We have $h(0) = 0$, $h(1) = -a$, and, by Lemma~\ref{lem:ppp}, 
\begin{equation*}
    h'(0^+) =  (p-1)(p-2-a) \geq  (p-1)(p-1-p^{\frac{p-2}{p-1}})>0.
\end{equation*}
 Moreover, $h$ is strictly concave on $[0,1]$, since for $\beta\in(0,1)$, 
 \begin{equation*}
     h''(\beta) =  (p-1)(p-2)(p-1-a)|\beta-1|^{p-4}(\beta-1) < 0.
 \end{equation*}

 If $0 < a < p^{\frac{p-2}{p-1}}-1$, then $h$ has a unique zero $\betaA = \betaA(p,a)$ in $[0,1]$. Moreover, $h((p-1-a)/p) >0$ (this turns out to be equivalent to $a < p^{\frac{p-2}{p-1}} -1$), so $\betaA > (p-1-a)/p$. The maximum of $k_{p,a}(0,\cdot)$ on $[0,1]$ is attained  at $\betaA$.
 
 If $ -1 < a \leq 0$, then $h$ is non-negative on  $[0,1]$ and thus the maximum of $k_{p,a}(0,\cdot)$ on $[0,1]$ is attained  at $\betaA=1$. This ends the proof. 
\end{proof}

We continue the proof of Theorem~\ref{thm:A}. 

\begin{proof}[\nextcaselabel{case:A-7}:  $p>2$ and $0 <a < p^{\frac{p-2}{p-1}}-1$]
Let $\betaA=\betaA(p,a)$ be the constant from the preceding lemma.
 This time, for $x\in[0,1]$, define
\begin{align*}
 A^p &=   k_{p,a}(0,\betaA) = 
 \frac{p\betaA-(p-1-a) + (p-1-a)|\betaA-1|^p}{(p-1-a)|\betaA|^p},\\
 v(x) &= |1-x|^p -  A^p |x|^p,\\
 u(x) &= \frac{p}{p-1-a} \bigl(\frac{p-1-a}{p} - x\bigr).
 \end{align*}
 
 By Lemma~\ref{lem:tech-A-2}, 
\[
 A^p \geq \frac{px - (p-1-a) + (p-1-a)(1-x)^p}{(p-1-a)x^p}.
\]
for $x\in(0,1]$. This is equivalent to $v\leq u$ on $(0,1]$ (and the inequality at $x=0$ is obvious). Thus, by Proposition~\ref{prop:majorize},  $A_{p,a}\leq A$. Since, by Lemma~\ref{lem:lower_bound}, also $A_{p,a}^p\geq k_{p,a}(0,\betaA) = A^p$, we conclude that  $A_{p,a}=A$. This ends the proof.
 \end{proof}
 
 \begin{remark}
  \label{rem:B=A-harder}
  We claim that for  $p>2$ and $0 < a < p^{\frac{p-2}{p-1}}-1$, if $A_{p,a}\geq 1$, then $B_{p,a}=A_{p,a}$. Indeed, the function 
\[
 h(\beta) = p-1-a - (p-1)\beta + (p-1-a)|\beta-1|^{p-2}(\beta-1), \quad  \beta \geq 0, 
\]
which we considered in the proof of Lemma~\ref{lem:tech-A-2} only for $\beta\in[0,1]$,
 is strictly concave on $[0,1]$ and strictly convex on $[1,\infty)$ with $\lim_{\beta\to\infty} h(\beta) = +\infty$. Thus $h$ changes the sign twice in $[0,\infty)$ (at $\beta_0$ and somewhere in $(1,\infty)$). Hence
\begin{align*}
 \sup \{ k_{p,a}(0,\beta) : \beta>0\} &=  \max\bigl\{k_{p,a}(0,\betaA), \lim_{\beta\to\infty}k_{p,a}(0,\beta)\bigr\}\\
 &= \max\{k_{p,a}(0,\betaA), 1\} = A_{p,a}^p.
\end{align*}
Thus, if $A_{p,a}\geq 1$, in the proof of Case~\ref{case:A-7} we get the majorization $v\leq u$ not only on $[0,1]$, but on $[0,\infty)$. Hence $B_{p,a} \leq A_{p,a}$ (by Proposition~\ref{prop:majorize}), and consequently  $B_{p,a} =A_{p,a}$.
 \end{remark}

\begin{proof}[\nextcase:  $p>2$ and $-1 <a\leq 0$]
 This time, for $x\in[0,1]$, define
\begin{align*}
 A^p &=  \frac{1+a }{p-1-a},\\
 v(x) &= |1-x|^p -  A^p |x|^p,\\
 u(x) &= \frac{p}{p-1-a} \bigl(\frac{p-1-a}{p} - x\bigr).
\end{align*}

By Lemma~\ref{lem:tech-A-2}, 
\[
 A^p \geq \frac{px - (p-1-a) + (p-1-a)(1-x)^p}{(p-1-a)x^p}.
\]
for $x\in[0,1)$. This is equivalent to $v\leq u$ on $[0,1)$ (and the inequality at $x=1$ is obvious). Thus, by Proposition~\ref{prop:majorize},  $A_{p,a}\leq A$. Since, by Lemma~\ref{lem:lower_bound}, also $A_{p,a}^p\geq k_{p,a}(0,1) = A^p$, we conclude that  $A_{p,a}=A$. This ends the proof.
\end{proof}

\setcounter{proofincases}{0} % resets counter of cases 

We have considered all cases and the proof of Theorem~\ref{thm:A} is finished.
We end with one more remark which will come handy in the next section. 

\begin{remark}
 \label{rem:transition}
  Fix $p>2$. We have $A_{p,a} = (\frac{1+a}{p-1-a})^p > 1$ for $a \geq p^{(p-2)/(p-1)} -1$ 
 and $A_{p,a} = \frac{1+a}{p-1-a} < 1$ for  $-1 < a \leq 0$.  We claim that there exists exactly one
 \[
  \aBA =\aBA(p) \in\bigl(0, p^{\frac{p-2}{p-1}}-1\bigr) 
 \]
such that $A_{p,\aBA} = 1$ (and, clearly, $A_{p,a}>1$ for $a>\aBA$) and, additionally, we want to identify this parameter.
 
Assume first that for some $a\in (0, p^{(p-2)/(p-1)}-1)$ we have $A_{p,a}=1$. Using the notation from  the proof of Case~\ref{case:A-7} above,  we have $A=1$ and $v(\betaA)=u(\betaA)$. This necessarily means that also $v'(\betaA)=u'(\betaA)$, i.e.,
\[
(1-\betaA)^{p-1} + \betaA^{p-1} = \frac{1}{p-1-a}
\]
(otherwise the majorization $v\leq u$ would not hold). Combining this with
\[
A^p =  \frac{p\betaA-(p-1-a) + (p-1-a)|\betaA-1|^p}{(p-1-a)|\betaA|^p} =1
\]
yields the necessary condition
\[
-(p-1)(1-\betaA)^p + p(1-\betaA)^{p-1} + (p-1)\betaA^{p} =1.
\]

  Therefore, let us  consider the function 
 \[
 g(\beta) = -(p-1)(1-\beta)^p + p(1-\beta)^{p-1} + (p-1)\beta^{p}, \quad \beta\in[0,1].
 \]
 We have $g(0) =1< p-1 =g(1)$ and
 \begin{align*}
  g'(\beta) &= p(p-1)(1-\beta)^{p-1} - p(p-1)(1-\beta)^{p-2} + p(p-1)\beta^{p-1}\\
  &= -p(p-1)(1-\beta)^{p-2}\beta + p(p-1)\beta^{p-1}\\
   &= p(p-1)\beta \bigl( - (1-\beta)^{p-2}+\beta^{p-2}\bigr).
\end{align*}
Hence $g$ decreases on $[0,1/2]$ and increases on $[1/2,1]$ and thus there exists exactly one $\betaBA=\betaBA(p)\in(1/2,1)$ such that $g(\betaBA)=1$.  
 
 Define $\aBA =\aBA(p)>0$ by the relation
 \[
  (1-\betaBA)^{p-1} + \betaBA^{p-1} = \frac{1}{p-1-\aBA}.
 \]
Let us denote $v(x) = |x-1|^p - |x|^p$ and let $u$ be the tangent to $v$ at $x=\betaBA$:
\[
 u(x) = -p\bigl((1-\betaBA)^{p-1} +\betaBA^{p-1}\bigr)(x-\betaBA) + (1-\betaBA)^{p} - \betaBA^p.
\]
The definitions of $\betaBA$, $\aBA$ imply that $u(x) = 1 -\frac{p}{p-1-\aBA} x$ and 
\[
\frac{(\betaBA-\frac{p-1-\aBA}{p}) + \frac{p-1-\aBA}{p}|\betaBA-1|^p}{\frac{p-1-\aBA}{p}|\betaBA|^p} =1.
\]
Since $\betaBA\in(1/2,1)$, this means that $\betaBA \geq \frac{p-1-\aBA}{p}$ and $k_{p,\aBA}(0,\betaBA) = 1$. Moreover, since $v$ is convex on $[0,1/2]$ and concave on $[1/2,\infty)$ and $v(0)=u(0)$, the majorization $v\leq u$ holds on $[0,1]$.      
Thus, by Proposition~\ref{prop:majorize}, $A_{p,\aBA}\leq 1$.  Since by Lemma~\ref{lem:lower_bound} also $A_{p,\aBA}\geq k_{p,\aBA}(0,\betaBA) = 1$, we conclude that  $A_{p,\aBA}=1$.
\end{remark}

\section{Estimates for \texorpdfstring{$H-I$}{H--I} on positive functions}
\label{sec:B}

We start with the range $p>2$, in which the proof of Theorem~\ref{thm:B} is simpler.

\begin{lemma}
\label{lem:Av1}
For $1\leq p  <\infty$, $a < p-1$ we have $B_{p,a} \geq \max\{A_{p,a}\indbr{a>-1}, 1\}$.
\end{lemma}

\begin{proof}
Clearly, for  $1\leq p  <\infty$  and $-1<a < p-1$, $B_{p,a} \geq A_{p,a}$.

We now prove that $B_{p,a}\geq 1$. To this end fix $\varepsilon>0$ and consider $f=\ind{[1,1+\varepsilon]}$. We have
\begin{equation*}
 Hf(t) = 
 \begin{cases}
  0 & \text{for } t\in[0,1],\\
  \frac{t-1}{t} & \text{for } t\in(1,1+\varepsilon),\\
  \frac{\varepsilon}{t} & \text{for } t\geq 1+\varepsilon.
  \end{cases}
\end{equation*}
Hence,
\begin{align*}
 \int_0^\infty |Hf(t) -f(t)|^p t^a dt 
 %%&= \int_1^{1+\varepsilon} t^{a-p} dt +\varepsilon^p \int_{1+\varepsilon}^\infty t^{a-p} dt\\
  &\geq \int_1^{1+\varepsilon} t^{a-p} dt= \frac{1}{1+a-p} \bigl( (1+\varepsilon)^{1+a-p}- 1\bigr), \\
%%&= \frac{1}{1+a-p} \bigl( (1+\varepsilon)^{1+a-p}- 1\bigr) -  \frac{\varepsilon^p}{1+a-p} (1+\varepsilon)^{1+a-p},\\
 \int_0^\infty |f(t)|^p t^a dt 
 &= %%\int_1^{1+\varepsilon} t^{a} dt=
 \begin{cases}
 \frac{1}{1+a} \bigl( (1+\varepsilon)^{1+a}- 1\bigr) & \text{if } a\neq  -1,\\
 \ln(1+\varepsilon) & \text{if } a= -1.
 \end{cases}
\end{align*}
To obtain the assertion it suffices to consider the ratio of these two quantities with $\varepsilon\to 0^+$.  
\end{proof}

We are ready to start the proof of Theorem~\ref{thm:B}. While the formula $B_{p,a} = \max\{A_{p,a}\indbr{a>-1}, 1\}$ can be derived abstractly, without identifying when $B_{p,a} = A_{p,a}$ and when $B_{p,a} =  1$, it is convenient to use Remark~\ref{rem:transition}.

\begin{proof}[\nextcaselabel{case:B-1}: $p\geq2$]
Fix $p\geq 2$. Let $\aBA =\aBA(p) \in (0, p^{(p-2)/(p-1)}-1)$
be the parameter identified in Remark~\ref{rem:transition}. For $a\in(\aBA,p-1)$ we have $A_{p,a} >1$, and $B_{p,a} = A_{p,a}$ follows from Remarks~\ref{rem:B=A-easy} and \ref{rem:B=A-harder}.

For $a\leq \aBA$ and $x\geq 0$ define
\begin{align*}
v(x) &= |1-x|^p - |x|^p,\\
u(x,a) &= \frac{p}{p-1-a}\bigl( \frac{p-1-a}{p} - x \bigr) = 1 -\frac{p}{p-1-a}x.
\end{align*}
Since $A_{p,\aBA} = 1$, $v(x) \leq u(x,\aBA)$ for $x\geq 0$ (by Remark~\ref{rem:B=A-harder}).
All the more, $v(x) \leq u(x,a)$ for $x\geq 0$ and $a\leq \aBA$. Thus, by Proposition~\ref{prop:majorize} and Lemma~\ref{lem:Av1}, $B_{p,a}=1$ for $a\leq \aBA$.
\end{proof}

\begin{proof}[\nextcase: $p=2$]
Up to obvious changes, the proof goes as in Case~\ref{case:B-1} above.
\end{proof}

\begin{proof}[\nextcase: $p=1$]
By the proof of Theorem~\ref{thm:C}, $C_{1,a} =\lim_{\beta\to\infty} k_{1,a}(0,\beta)$. Since the latter quantity is a lower bound for $B_{1,a}$ (by Lemma~\ref{lem:lower_bound}), we conclude that $B_{1,a}= C_{1,a}$. 
\end{proof}

\begin{lemma}
 \label{lem:tech-B-2}
Suppose that $1 < p < 2$ and $a \leq p-2$. Then there exists exactly one number $\betaB=\betaB(p,a)> (p-1-a)/p$
such that
\[
 \sup \{ k_{p,a}(0,\beta) : 0< \beta  \}
 = \sup \{ k_{p,a}(0,\beta) : (p-1-a)/p < \beta  \}
 = k_{p,a}(0,\betaB).
\]
Moreover, $\betaB=\betaB(p,a)$ is the unique solution of the equation
\begin{equation*}
-(p-1-a)|1-\betaB|^{p-2}(\betaB-1)-(p-1)\betaB +p-1-a =0
\end{equation*}
in the interval $(1,\infty)$.
\end{lemma}

\begin{proof}
Like in the proof of Lemma~\ref{lem:tech-A-2}, in order to maximize $k_{p,a}(0,\cdot)$ we investigate the function 
\[
 h(\beta) = p-1-a - (p-1)\beta + (p-1-a)|\beta-1|^{p-2}(\beta-1), \quad  \beta\geq 0,
 \]
 which is  of the same sign as the derivative $\frac{d}{d\beta} k_{p,a}(0,\beta)$.
 We have $h(0) = 0$, $h(1) = -a>0$,
\[
   h'(0^+) =  (p-1)(p-2-a) \geq0,
   \]   
    and, for $\beta\in[0,\infty)\setminus\{1\}$, 
  \[
    h''(\beta) =  (p-1)(p-2)(p-1-a)|\beta-1|^{p-4}(\beta-1).
  \]
 Thus $h$ is strictly convex on $[0,1]$ and strictly concave on $[1,\infty)$. Since $h(\beta)\to -\infty$ as ${\beta\to\infty}$, we conclude that $h$ has a unique zero $\betaB = \betaB(p,a)$ in $[1,\infty)$. Moreover, \begin{align*}
 h((p-1-a)/p) &= (p-1-a)\Bigl( 1 - \frac{p-1}{p} - \frac{|1+a|^{p-2}(1+a)}{p^{p-1}}\Bigr) \\
 &\geq (p-1-a)\Bigl( \frac{1}{p} - \frac{(p-1)^{p-1}}{p^{p-1}}\Bigr) > 0
 \end{align*}
 (by Lemma~\ref{lem:ppp}), so $\betaB > (p-1-a)/p$. The maximum of $k_{p,a}(0,\cdot)$ on $[0,\infty)$ is attained  at $\betaB$.
\end{proof}

We continue the proof of Theorem~\ref{thm:B}. 

\begin{proof}[\nextcase:  $1 < p < 2$ and $a \leq p-2$]
Let $\betaB=\betaB(p,a)$ be the constant from the preceding lemma.
 This time, for $x\geq 0$, define
\begin{align*}
 B^p &=   k_{p,a}(0,\betaB) = 
 \frac{p\betaB-(p-1-a) + (p-1-a)|\betaB-1|^p}{(p-1-a)|\betaB|^p},\\
 v(x) &= |1-x|^p -  B^p |x|^p,\\
 u(x) &= \frac{p}{p-1-a} \bigl(\frac{p-1-a}{p} - x\bigr).
 \end{align*}
 
 By Lemma~\ref{lem:tech-B-2}, 
\[
 B^p \geq \frac{px - (p-1-a) + (p-1-a)|1-x|^p}{(p-1-a)x^p}.
\]
for $x>0$. This is equivalent to $v\leq u$ on $(0,\infty)$ (and the inequality at $x=0$ is obvious). Thus, by Proposition~\ref{prop:majorize},  $B_{p,a}\leq B$. Since, by Lemma~\ref{lem:lower_bound}, also $B_{p,a}^p\geq k_{p,a}(0,\betaB) = B^p$, we conclude that  $B_{p,a}=B$. This ends the proof.
 \end{proof}

 We remark that for $a<p-2$ we have $v'(0) < u'(0)$ in the above proof. Thus the majorization $v\leq u$ does not extend to the whole real line. This means that $B_{p,a}<C_{p,a}$ for  $a< p-2$,.
 
What follows is very close to the reasonings in Lemma~3.3 and Proposition~3.7 in \cite{MR3558516}. 

\begin{lemma}
 \label{lem:sup-k-BC}
 Suppose that $1 < p < 2$ and $p-2 < a < p-1$. Denote 
 \[
  s_{p,a} \coloneqq \sup\{  k_{p,a}(\alpha,\beta) : 0\leq  \alpha < (p-1-a)/p <  \beta\}.
 \]
 Then $s_{p,a}\geq ((1+a)/(p-1-a))^p$ and  $s_{p,a} > 1$.
Moreover, if $s_{p,a} > ((1+a)/(p-1-a))^p$, then there exists a point
 \[
  (\alphaBbis,\betaBbis) =   \bigl(\alphaBbis(p,a),\betaBbis(p,a)\bigr) \in \bigl(0,(p-1-a)/p\bigr) \times \bigl((p-1-a)/p,\infty\bigr)
 \]
 such that $s_{p,a} = k_{p,a}(\alphaBbis,\betaBbis)$.
\end{lemma}

\begin{proof}
 First note that $a>p-2>-1$ and hence $(p-1-a)/p<1$. We have
 \[
  s_{p,a} \geq \lim_{\alpha\to \frac{p-1-a}{p}^-} k_{p,a}(\alpha,1) = \Bigl(\frac{1+a}{p-1-a}\Bigr)^p
 \]
 and, for $n$ large enough,
 \[
  s_{p,a} \geq k_{p,a}(0,n) = \frac{pn - (p-1-a)  + (p-1-a)|n-1|^p}{(p-1-a)|n|^p} >1,
 \]
(the sharp inequality holds since $p<2$).

As for the second part, let $(\alpha^{(n)},\beta^{(n)})$, ${n\geq 1}$, be a sequence of points realizing the value of the supremum $s_{p,a}$. By compactness, we may and do assume that $\alpha^{(n)}\to \alpha^{(\infty)}\in [0,(p-1-a)/p]$. Without loss of generality, also $\beta^{(n)}\to\beta^{(\infty)}\in [(p-1-a)/p,+\infty]$. We shall now exclude several different cases in which $s_{p,a}>((1+a)/(p-1-a))^p$ cannot hold. It is useful to notice that $k_{p,a}(\alpha,\beta)$ (with $\alpha>0$) can be written as a convex combination:
\begin{equation}
 \label{eq:k-conv-comb}
 k_{p,a}(\alpha,\beta) = w_1(\alpha,\beta)\cdot|1-1/\alpha|^p + w_2(\alpha,\beta)\cdot |1-1/\beta|^p,
\end{equation}
where
\begin{align*}
w_1(\alpha,\beta) &= \frac{(\beta-\frac{p-1-a}{p})|\alpha|^p}{(\beta-\frac{p-1-a}{p})|\alpha|^p + (\frac{p-1-a}{p}-\alpha)|\beta|^p}, \\
w_2(\alpha,\beta) &= \frac{ (\frac{p-1-a}{p}-\alpha)|\beta|^p}{(\beta-\frac{p-1-a}{p})|\alpha|^p + (\frac{p-1-a}{p}-\alpha)|\beta|^p}.
\end{align*}

If $\beta^{(\infty)} = (p-1-a)/p$, then $k_{p,a}(\alpha^{(n)},\beta^{(n)})\to ((1+a)/(p-1-a))^p$ (regardless of whether $\alpha^{(\infty)} = (p-1-a)/p$ or not) and $s_{p,a}=((1+a)/(p-1-a))^p$.

If $\beta^{(\infty)} = +\infty$, then $\alpha^{(\infty)} = (p-1-a)/p$ (otherwise, by \eqref{eq:k-conv-comb}, we would have $s_{p,a}=1$). Again by \eqref{eq:k-conv-comb}, we conclude
that
\[
 \min\{1,((1+a)/(p-1-a))^p\} \leq s_{p,a} \leq \max\{1,((1+a)/(p-1-a))^p\},
\]
which by the first part of the lemma yields $s_{p,a}=((1+a)/(p-1-a))^p$. 

Hence, let us consider the case  $(p-1-a)/p < \beta^{(\infty)} < \infty$.
%, then we can take $(\alphaBbis,\betaBbis) = (\alpha^{(\infty)},\beta^{(\infty)})$. 
Note that inevitably $\alpha^{(\infty)}>0$, since $ \frac{d}{d\alpha} k_{p,a}(0^+,\beta) $ is of the same sign as 
\[
-(p\beta- (p-1-a))(p-2-a)\beta^p >0.
 \]
If $\alpha^{(\infty)}$ is equal $(p-1-a)/p$, then  $s_{p,a}=((1+a)/(p-1-a))^p$. Thus, if $s_{p,a}>((1+a)/(p-1-a))^p$ then $\alpha^{(\infty)}< (p-1-a)/p$  and we can take $(\alphaBbis,\betaBbis) = (\alpha^{(\infty)},\beta^{(\infty)})$. This ends the proof.
\end{proof}

We are ready for the last case in the proof of Theorem~\ref{thm:B}.

\begin{proof}[\nextcase: $1 < p < 2$ and $p-2 < a < p-1$]
Fix $1 < p < 2$ and $p-2 < a < p-1$. We shall consider two subcases.

\emph{Subcase 1.} Suppose first that $s_{p,a}>((1+a)/(p-1-a))^p$ and let $(\alphaBbis,\betaBbis)$ be the point identified in  Lemma~\ref{lem:sup-k-BC}. Denote 
\begin{align*}
 B^p &= k_{p,a}(\alphaBbis,\betaBbis),\\
 v(x) &= |x-1|^p - B^p |x|^p,\\
 u(x) &= D(p-1-a - px),
 \end{align*}
 where
$(\alphaBbis,\betaBbis)$ are defined in Lemma~\ref{lem:sup-k-BC} and
\[
 D\coloneqq \frac{k_{p,a}(\alphaBbis,\betaBbis)|\betaBbis|^p -|\betaBbis-1|^p }{p\betaBbis - (p-1-a)} = -\frac{k_{p,a}(\alphaBbis,\betaBbis)|\alphaBbis|^p -|\alphaBbis-1|^p }{p-1-a - p\alphaBbis} 
\]
(both numbers are equal  by the definition of $k_{p,a}$).

For every $x\in[0, (p-1-a)/p)$,
\begin{align*}
k_{p,a}(\alphaBbis,\betaBbis) 
&\geq
k_{p,a}(x,\betaBbis) \\
&=
\frac{(p\betaBbis - (p-1-a))|x-1|^p + (p-1-a - px)|\betaBbis-1|^p }{(p\betaBbis - (p-1-a))|x|^p + (p-1-a - px)|\betaBbis|^p},
\end{align*}
which is equivalent to
\begin{equation}
\frac{k_{p,a}(\alphaBbis,\betaBbis)|\betaBbis|^p -|\betaBbis-1|^p }{p\betaBbis - (p-1-a)}(p-1-a -px) \geq |x-1|^p - B^p |x|^p.
\end{equation}
For every $x\in((p-1-a)/p,\infty)$,
\begin{align*}
k_{p,a}(\alphaBbis,\betaBbis) 
&\geq
k_{p,a}(\alphaBbis,x) \\
&=
\frac{(px - (p-1-a))|\alphaBbis-1|^p + (p-1-a - p\alphaBbis)|x-1|^p }{(px - (p-1-a))|\alphaBbis|^p + (p-1-a - p\alphaBbis)|x|^p},
\end{align*}
which is equivalent to
\begin{equation}
-\frac{k_{p,a}(\alphaBbis,\betaBbis)|\alphaBbis|^p -|\alphaBbis-1|^p }{p-1-a - p\alphaBbis}(p-1-a -px) \geq |x-1|^p - B^p |x|^p.
\end{equation}
We conclude that $u(x) \geq v(x)$ 
for $x\geq 0$ (for $x=(p-1-a)/p$ the inequality holds, since $B\geq (1+a)/(p-1-a)$).
Hence, by Proposition~\ref{prop:majorize}, $B_{p,a}\leq B$. By Lemma~\ref{lem:lower_bound} and the definition of $B$, we conclude that $B_{p,a}=B$.

Moreover, the fact that
\[
 \frac{d}{d\alpha} k_{p,a}(\alpha,\beta)\bigg\rvert_{(\alpha,\beta) = (\alphaBbis,\betaBbis)} 
 = \frac{d}{d\beta} k_{p,a}(\alpha,\beta)\bigg\rvert_{(\alpha,\beta) = (\alphaBbis,\betaBbis)}
 = 0 
\]
implies that $u$ is tangent to $v$ at $x\in \{\alphaBbis,\betaBbis\}$.
Since $B>1$, Lemma~\ref{lem:ccc} yields that $v$ is first, concave, then convex, then again concave. Thus, by Lemma~\ref{lem:rolle}, the majorization $v\leq u$ extends to the whole line, so $C_{p,a} = B_{p,a} = B$. This ends the proof in the first subcase.

\emph{Subcase 2.} Suppose now that in Lemma~\ref{lem:sup-k-BC}  we have   $s_{p,a}=((1+a)/(p-1-a))^p$. Denote 
\begin{align*}
 B &=  \frac{1+a}{p-1-a},\\
 v(x) &= |1-x|^p -  B^p |x|^p,\\
 u(x) &= \frac{(1+a)^{p-1}}{p^{p-3}(p-1-a)} \bigl(\frac{p-1-a}{p} - x\bigr).
\end{align*}
The function $u$ is tangent to $v$ at $x=(p-1-a)/p$.

We claim that $v(x) \leq u(x)$ for $x\geq 0$. Indeed, suppose by contradiction that 
there exists some $x_0\geq 0$, such that
\begin{equation}\label{ineq:ban-jan-not}
|x_0-1|^p -B^p|x_0|^p > \frac{(1+a)^{p-1}}{p^{p-3}(p-1-a)} \bigl(\frac{p-1-a}{p} - x\bigr).
\end{equation}
Of course we cannot have $x_0 = \frac{p-1-a}{p} $. Suppose first, that $x_0 > \frac{p-1-a}{p} $. Since
\begin{equation*}
\lim_{x\to\frac{p-1-a}{p} } \frac{|x-1|^p - B^p|x|^p}{x-\frac{p-1-a}{p} } = -\frac{(1+a)^{p-1}}{p^{p-3}(p-1-a)},
\end{equation*}
we conclude from~\eqref{ineq:ban-jan-not}, that for some 
\[
(\alpha,\beta)\in \bigl(0,(p-1-a)/p\bigr)\times\bigl((p-1-a)/p,\infty\bigr)
\]
we have
\begin{equation*}
|\beta-1|^p - B^p|\beta|^p > \frac{|\alpha-1|^p - B^p|\alpha|^p}{\alpha-\frac{p-1-a}{p} } (\beta-\frac{p-1-a}{p} )
\end{equation*}
(it suffices to take $\beta=x_0$ and $\alpha$ smaller than, but close to $\frac{p-1-a}{p} $) or equivalently
\begin{equation*}
%\frac{(\beta-\frac{p-1-a}{p} )|\alpha-1|^p + (\frac{p-1-a}{p} -\alpha)|\beta-1|^p}{(\beta-\frac{p-1-a}{p} )|\alpha|^p + (\frac{p-1-a}{p} -\alpha)|\beta|^p} > (\frac{p-1-a}{p} ^{-1}-1)^p.
k_{p,a}(\alpha,\beta) > \Bigl(\frac{1+a}{p-1-a}\Bigr)^p
\end{equation*}
We arrive at the same conclusion, if $x_0 < \frac{p-1-a}{p} $ (it suffices to take $\alpha = x_0$ and $\beta$ greater than, but close to $\frac{p-1-a}{p} $). 
This finishes the proof of the claim.

The claim together with Proposition~\ref{prop:majorize}, implies that $B_{p,a}\leq B$. By Lemma~\ref{lem:lower_bound} and the definition of $B$, we conclude that $B_{p,a}=B$.

Moreover,  
\[
 \frac{p-1-a}{p} \leq \frac{1}{1+B^{p/(p-2)}}.
\]
(since $B = (p-1-a)/p>1$ by Lemma~\ref{lem:sup-k-BC}). 
Since $u$ is tangent to $v$ at $(p-1-a)/p$ and,by Lemma~\ref{lem:ccc}, $v$ is concave on $\bigl(-\infty,{1}/({1+B^{p/(p-2)}})\bigr)$, we conclude that the majorization $v\leq u$ extends to the whole line, so $C_{p,a} = B_{p,a} = B$. This ends the proof in the second subcase.
\end{proof}

We have considered all cases and the proof of Theorem~\ref{thm:B} is finished.
\setcounter{proofincases}{0} % resets counter of cases 

\section{Comparisons of \texorpdfstring{$H$}{H} and its adjoint}
\label{sec:Hstar}

We start with the following observation.

\begin{lemma}
\label{thm:A-reversed}
For $1\leq p<\infty$, $-1<a<p-1$, and any positive decreasing function $\varphi\colon [0,\infty)\to \RR$ such that $\lim_{t\to \infty}\varphi(t)=0$, we have
\begin{equation}
\label{eq:thm-A-reversed}
  \int_0^\infty |\varphi (t)|^p t^a dt \leq A_{p,p-2-a}^p \int_0^\infty |H\varphi(t) - \varphi(t)|^p t^a dt.
\end{equation}
\end{lemma}

\begin{proof}
By a reasoning as in Proposition~\ref{prop:majorize}, it suffices to prove that
\begin{equation}
 \label{eq:maj_p-2-a}
  |x|^p - A_{p,p-2-a}^p |1 - x|^p \leq D \Bigl( \frac{p-1-a}{p} - x\Bigr).
\end{equation}
for some constant $D$ and all $x\in[0,1]$. 
Note that if $\varphi$ is a function like above and $\int_0^\infty |H\varphi(t) - \varphi(t)|^p t^a dt<\infty$, then also
\begin{equation*}
  \int_0^\infty |\varphi(t)|^p t^a dt<\infty,\quad  \int_0^\infty |H\varphi(t)|^p t^a dt<\infty,
\end{equation*}
see, e.g., Proposition~7.12 and Lemma~4.5 in \cite{MR928802} (their premises are fulfilled, since our assumptions imply that $\lim_{t\to \infty} H\varphi(t)=0$ and $p/(a+1)>1$). Therefore the integration by parts from the proof of Proposition~\ref{prop:majorize} can be carried out legally.

Substituting $ \mathrm{x} = 1-x$ transforms \eqref{eq:maj_p-2-a} into 
\[
  |1-\mathrm{x}|^p - A_{p,p-2-a}^p |\mathrm{x}|^p \leq -D \Bigl( \frac{p - 1 - (p-2-a)}{p} -\mathrm{x}\Bigr), \quad \mathrm{x}\in[0,1],
\]
which is exactly the inequality we proved during the proof of Theorem~\ref{thm:B} (note that $-1< p-2-a <p-1$).
This finishes the proof.
\end{proof}

We can now proceed as in \cite{MR3180926} to prove Corollary~\ref{cor:E}. 
For the sake of completeness and for the reader's convenience, we recall the reasoning in detail.

\begin{proof}[Proof of Corollary~\ref{cor:E}] Fix $1\leq p<\infty$ and $-1<a<p-1$. If $f\colon [0,\infty)\to \RR$ is positive and $ \int_0^\infty |H^*f(t)|^p t^a dt < \infty$, then the function
\[
 \varphi(t):= H^* f(t) %= \int_t^\infty \frac{f(s)}{s} ds
\]
is finite, positive, and decreasing. By Fubini's theorem, 
\begin{equation*}
 H\varphi(t) = \frac{1}{t} \int_0^t \int_s^\infty \frac{f(u)}{u} du\ ds = \frac{1}{t} \int_0^\infty \min\{u,t\} \frac{f(u)}{u} du = Hf(t) + \varphi(t).  
\end{equation*}
Thus
\[
H \varphi(t) - \varphi(t) =  H f(t)
\]
and the left-hand side inequality in \eqref{eq:cor:E} follows immediately from Theorem~\ref{thm:A}.

% Note: above we have $\int_0^\infty |Hf(t)|^p t^a dt < \infty$, since $Hf(t)\leq H(H^*f)(t)$ and we can apply Hardy's inequality \eqref{eq:Hardy}).

Similarly,  to prove the right-hand side inequality in \eqref{eq:cor:E} we use Lemma~\ref{thm:A-reversed}.
Indeed, if $f\colon [0,\infty)\to \RR$ is positive and $ \int_0^\infty |Hf(t)|^p t^a dt < \infty$, then by Fubini's theorem, 
\begin{align*}
 H^* f(t) &=  \int_t^\infty \frac{1}{u^2} \int_t^u f(s)ds\  du = \int_t^\infty \Bigl( \frac{1}{u} Hf(u) du - \frac{1}{u^2} \int_0^t f(s)ds\Bigr) du \\
&=  H^*\bigl( Hf\bigr)(t) - Hf(t)
\end{align*}
By H\"older's inequality,
\[
 \int_t^\infty H f(u) \frac{du}{u} 
 \leq \Bigl(\int_t^\infty |Hf(u)|^p u^a du\Bigr)^{1/p}  \Bigl(\int_t^\infty u^{ \frac{(-a/p -1)p}{p-1}} du\Bigr)^{p/(p-1)}
 <\infty
\]
(recall that $a>-1$). Thus, $H^*\bigl( Hf \bigr)(t)$, and consequently also $H^* f(t)$, is finite for any $t>0$. Therefore, we can define $\varphi$ as before and then use Lemma~\ref{thm:A-reversed}.

The constants are optimal since the constants from Theorem~\ref{thm:A} and Lemma~\ref{thm:A-reversed}  are best possible and the reasoning from the beginning of the proof can be reversed:
for a positive decreasing locally absolutely continuous $\varphi:[0,\infty)\to\RR$ with $\lim_{t\to \infty}\varphi(t)=0$ one can set
\[
 f(t) \coloneqq -t\varphi'(t),
\]
so that $H^*f(t) = \varphi(t)$ (cf.\ \cite{MR3180926}).
\end{proof}

\begin{proof}[Proof of Proposition~\ref{prop:F}]
By a reasoning as in the proof of Corollary~\ref{cor:E}, $F_{p,a}\leq C_{p,a}$ (c.f. \cite[Lemma 1.2]{MR3868629}). Below we construct a family of functions which shows that $F_{p,a}\geq C_{p,a}$.

Fix $1\leq p<\infty$, $a<p-1$. Fix also $\alpha < (p-1-a)/p< \beta$ and for $n\in \NN$, consider the function
\begin{equation*}
 f_n(t) = \beta(\beta-1) t^{\beta-1} \ind{[0,1]}(t) + n(\alpha-\beta)\ind{[1-1/n,1]}(t) + \alpha(\alpha-1) t^{\alpha-1}\ind{1,\infty}(t). 
\end{equation*}
We have
\begin{equation*}
 Hf_n(t) = 
 \begin{cases}
  (\beta-1) t^{\beta-1} & \text{ for } t\in [0,1-1/n],\\
   (\beta-1) t^{\beta-1} + n(\alpha-\beta)(t-1+1/n)\frac{1}{t} & \text{ for } t\in (1-1/n,1],\\
    (\alpha-1) t^{\alpha-1}   & \text{ for } t\in (1,\infty),\\
 \end{cases}
\end{equation*}
and
\begin{equation*}
 H^*f_n(t) = 
 \begin{cases}
 -\alpha t^{\alpha-1}   & \text{ for } t\in (1,\infty),\\
    -\alpha - n(\alpha-\beta)\ln(t) +\beta-\beta t^{\beta-1} & \text{ for } t\in (1-1/n,1],\\
 -\beta t^{\beta-1} +(\beta-\alpha)(1+n\ln(1-1/n))& \text{ for } t\in [0,1-1/n].\\
 \end{cases}
\end{equation*}

On the intervals $[1-1/n,1]$, $n\geq 2$, the functions $ Hf_n$, $ H^*f_n$, and the weight $t^a$ are bounded. Hence,
\begin{align*}
 \lim_{n\to\infty} \int_0^\infty |Hf_n(t)|^p t^a dt &= |\beta-1|^p \int_0^1 t^{p(\beta-1)+a} dt  + |\alpha-1|^p \int_1^\infty t^{p(\alpha-1)+a} dt\\
 &= \frac{|\beta-1|^p}{p(\beta-1)+a+1} -\frac{|\alpha-1|^p}{p(\alpha-1)+a+1},\\
  \lim_{n\to\infty} \int_0^\infty |H^*f_n(t)|^p t^a dt &= |\beta|^p \int_0^1 t^{p(\beta-1)+a} dt  + |\alpha|^p \int_1^\infty t^{p(\alpha-1)+a} dt\\
  &= \frac{|\beta|^p}{p(\beta-1)+a+1} -\frac{|\alpha|^p}{p(\alpha-1)+a+1}.
\end{align*}
The ratio of those two quantities is a lower bound for $F_{p,a}^p$. Taking the supremum over $\alpha<(p-1-a)/p<\beta$ and using Theorem~\ref{thm:C} yields $F_{p,a}\geq C_{p,a}$. 
\end{proof}

\section*{Acknowledgements}

I thank Santiago Boza and Javier Soria for a conversation which initiated this research.

  \bibliographystyle{amsplain}
  \bibliography{H-I}

\end{document}